\documentclass{article}
\usepackage{graphicx} 
\usepackage{amsmath}

\usepackage[utf8]{inputenc}
\usepackage{amsmath}
\usepackage{amsthm}
\usepackage{amssymb}
\usepackage{authblk}
\usepackage{mathtools}
\usepackage{xcolor}
\usepackage[hidelinks]{hyperref}
\usepackage[numbered]{bookmark}
\usepackage{enumerate}
\usepackage{amsfonts,amssymb,amsmath,amsthm,bbm,mathrsfs}
\usepackage{graphicx, color, psfrag,hyperref,xcolor,bbm}
\usepackage{mathrsfs} 
\usepackage{tikz-cd}
\usepackage{comment}
\usepackage{hyperref}

\usepackage[numbers]{natbib}

\usepackage{cleveref}
  \crefname{section}{Section}{Sections}
  \crefname{figure}{Figure}{Figures}
  \crefname{theorem}{Theorem}{Theorems}
  \crefname{lemma}{Lemma}{Lemmas}
  \crefname{proposition}{Proposition}{Propositions}  
  \crefname{corollary}{Corollary}{Corollaries}
  \crefname{definition}{Definition}{Definitions}
  \crefname{example}{Example}{Examples}
  \crefname{remark}{Remark}{Remarks}

\newtheorem{theorem}{Theorem}[section]

\newtheorem{lemma}[theorem]{Lemma}

\newtheorem{proposition}[theorem]{Proposition}

\newtheorem{definition}[theorem]{Definition}

\newtheorem{remark}[theorem]{Remark}

\DeclarePairedDelimiter\floor{\lfloor}{\rfloor}

\title{{ On the analytic extension of  Random Riemann Zeta Functions for some probabilistic models of the primes}}

 \author{Vlad Margarint \thanks{University of North Carolina at Charlotte, vmargari@charlotte.edu; \\URL: \href{https://margarintvlad.com/}{margarintvlad.com}} \hspace{2mm} and Stanislav Molchanov \thanks{University of North Carolina at Charlotte, smolchan@charlotte.edu;\\URL: \href{https://math.charlotte.edu/directory/stanislav-molchanov/}{Website}}}
\date{October 2024}

\begin{document}
\maketitle 
\begin{abstract}
 The first step in the formulation and study of the Riemann Hypothesis is the analytic continuation of the Riemann Zeta Function (RZF) in the full Complex Plane with a pole at $s=1$. In the current work, we study the analytic continuation of two random versions of RZF using, for $Re s>1$, the Euler representation of ZF in terms of the product of functions over primes. In the first case, we substitute in the Euler product pseudo-prime numbers from the famous Cram\'er Model. In the second case, we use pseudo-primes with local symmetries.  We show that in the Cram\'er case analytic continuation is possible $\mathbb{P}$-a.s. for $Res>1/2$, but not through the critical line $Re s=1/2.$ In the second case, we show that the analytic continuation is possible in a larger domain. We also study for the Cram\'er pseudo-primes several problems from Additive Number Theory.
\end{abstract}

\section{Introduction}

The idea that the primes are randomly distributed in the series of natural numbers in explicit form was formulated in the beginning of the 19th century (Legendre, Gauss) but only H. Cram\'er around 100 years later (1936) published the paper (\citep{Cramer}) where the primes were included as an element of the probability space of random sequences. Namely, Cram\'er (\citep{Cramer}) called each natural number $n \geq 4$ quasi-prime with probability $\frac{1}{\ln n}$ independently on other $n'\neq n.$

In this paper, we are interested in studying two probabilistic models of the primes and the corresponding Random Riemann Zeta functions associated with these models. 
For recent works in which new probabilistic models of the primes are developed, see for example \citep{Tao}. 


This problem we study is linked with the famous Riemann Hypothesis. This famous problem concerns the non-trivial zeroes of the Riemann Zeta function that are conjectured to all be on the critical line $Re s=1/2.$ 

One of the important steps in the study of this problem is the analytic continuation of the Riemann Zeta function, which is a-priori defined only on the region $Re s>1$ in the Complex Plane. In this region of the Complex Plane The Riemann Zeta function is linked with a certain product of primes through the Euler Formula (see \citep{tit}).

In this work, we consider probabilistic versions of the Riemann Zeta Function which are presented as a product of fractions similar to the ones from Euler product but not associated with the true primes but with probabilistic models of the primes. The first such model that we consider consists of Cram\' er's quasi-primes. In the first part of the paper, we study the analytic continuation of this Random Zeta function. We provide a strategy to show the analytic continuation of this function up to $Re s =1/2$ for almost every realization of the quasi-primes. We also study the problem of the impossibility of analytic continuation of this function through the critical line $Re s =1/2$. 

Our method relies on the use of the Kolmogorov's Two-Series Theorem along with a technical lemma. These techniques are applied after we obtain a representation of the corresponding Random Riemann Zeta function in terms of the complex power of quasi primes.  
In the following section, we show how one can apply techniques from Probability Theory, namely the Borel-Cantelli Lemmas, to study some problems in the spirit of the ones from Number Theory in the framework of the quasi primes.

In the last section, in order to contrast it with the Cram\'er Model we introduce a model of pseudo primes that possess correlations and exhibit local symmetries. We study the analytic continuation of the corresponding Random Riemann Zeta function and we conclude that the region on which one can analytically extend this Random Riemann Zeta function is larger than the region of the corresponding Random Riemann Zeta function for the Cram\'er Model.





The paper is organized as follows. In the first section after the Introduction, we discuss the Cram\'er Model of the primes, we define the corresponding Random Riemann Zeta function, and we state our main result. In Section 3 we prove our main Theorem and in Section 4 we comment on an application of Probability Theory techniques in the study of some Number Theory problems phrased in the language of the Cram\'er quasi primes. In Section 5 we introduce a model of pseudo primes that, in contrast to the Cram\'er Model, has correlations and exhibits local symmetries. We then study the analytic extension of the Random Riemann Zeta Function corresponding to this model. 
\vspace{2mm}

\textbf{Acknowledgements:}
The authors would like to thank Academician I. A. Ibragimov, Md. Kashif Jamal, D. Beliaev and James Rickards for useful discussions. VM acknowledges the support of the UNCC FRG Grant Fund No. 111432.

\section{The Cram\'er Model of the Primes and the corresponding Random Riemann Zeta Function}
Cram\'er (\citep{Cramer}) called the numbers $n \geq 4$ \textit{quasi-primes} with probability $\frac{1}{lnn}$ independently, for different $n$.  Let $\Pi(w)$ denote the set of quasi-primes.

We include $2$ and $3$ in $\Pi(w)$, that is we have 
$$\Pi(\omega)=\{2<3 <p_3(\omega)<p_4(\omega)<\ldots \}.$$

Let us also consider $\Pi(n, \omega)=\#\{p_i(\omega) \leq n \}$, for $n \geq 4.$

The joint distribution of $p_i$, $i \geq 1$ is explicit. 
For example, $$\mathbb{P}(p_2=3, p_3=6, p_4=11, p_5=12)=1\cdot\prod_{i=4}^5\left(1-\frac{1}{\ln i}\right)\frac{1}{\ln6}\cdot \prod_{i=7}^{10}\left(1-\frac{1}{\ln i}\right)\frac{1}{\ln11}\cdot \frac{1}{\ln12}.$$

Since $\Pi(n, \omega)=\#\{p_i(\omega) \leq n \}=2+\sum_{k=4}^n\epsilon_k$, with $\epsilon_k=1$, with probability $\frac{1}{\ln k}$ and $\epsilon_k=0$ with probability $1-\frac{1}{\ln k}$, for $n \geq 4$, we have 

\begin{equation}
\mathbb{E}[\Pi(n, \omega)]=2+\sum_{k=4}^n\frac{1}{\ln k}.
\end{equation}
We note that $\epsilon_k$ is a Bernoulli sequence with slowly decreasing probability of success. 

We further obtain that 

$$\mathbb{E}[\Pi(n, \omega)]=2+\sum_{k=4}^n\frac{1}{\ln k}=\int_{e}^n\frac{dx}{ln x}+O(1)=Li(n)+O(1).$$

We note that we use $Li(n)=\int_{e}^n\frac{dx}{\ln x},$ for $n\geq e.$

Similarly, we have that $$B_n^2=Var[\Pi(n, \omega)]=\sum_{k=4}^n\frac{1}{\ln k}\left(1-\frac{1}{\ln k}\right)=\frac{n}{\ln n}-\frac{3}{2}\frac{n}{\ln^2n}+O\left(\frac{n}{\ln^3n}\right).$$

Using that $\epsilon_k$, $k \geq 1$, is bounded, and the fact that $B_n\to \infty$, we obtain from the Kolmogorov's form of the Law of the Iterated Logarithm that

\begin{equation}
\limsup_{n \to \infty}\frac{|\Pi(n, \omega)-\mathbb{E}[\Pi(n, \omega)]|}{B_n\sqrt{2\ln \ln n}}=1,
\end{equation}
i.e. $\Pi(n,\omega)=Li(n)+O(\sqrt{n})$.

It is known that (see, \citep{tit}) a similar relation
$$ |\pi(n)-Li(n)|=O(\sqrt{n}\ln n)$$ 
for true primes is equivalent to the fundamental Riemann conjecture on the zeros of the $\zeta$ function:

\begin{equation}\label{Eulerformula1}
\zeta(s)=\sum_{n=1}^{\infty}\frac{1}{n^s}=\prod_{n=1}^{\infty}\frac{1}{(1-\frac{1}{p_n^s})}, 
\end{equation}
where $\{p_n, n\geq 1\}=\{2,3,5,\ldots \}$ is the set of true primes. Euler discovered the formula \eqref{Eulerformula} and used it only for real $z$. Riemann studied the $\zeta(\cdot)$ function for complex $s=\sigma+it \in \mathbb{C}.$

In particular, he constructed an analytic continuation of $\zeta(s)$ which is defined initially only for $Re s > 1$, and proved the functional equation connecting $\zeta(s)$, $\zeta(1-s)$ and the Gamma function $\Gamma(s)$. In addition, he formulated his conjecture (see, for example, \citep{tit}).

Using the similarity with \eqref{Eulerformula1}, we define for $\pi(\omega)$ the corresponding random $\zeta=\zeta(s, \omega)$ function by the formula 

\begin{equation}
\zeta(s, \omega)=\prod_{n=1}^{\infty}\frac{1}{(1-\frac{1}{p_n^s(\omega)})}
\end{equation}

in the beginning for $Re s >1.$ In this region, the infinite product converges $\mathbb{P}$-a.s. (since $p_n(\omega)\sim n\ln n$ $\mathbb{P}$-a.s.) and $|\zeta(s, \omega)|>0,$ i.e. there are no zeros. Our goal is to construct the analytic continuation of the $\zeta(s, \omega)$ in a larger domain and prove the following.

\begin{theorem}
a) $\zeta(s, \omega)$ has analytic continuation on the region $Re s> 1/2$ except at a singular point $s_0=1$, which is the simple pole.\\
b) For any fixed $t \in \mathbb{R},$ we have that $\limsup_{\sigma \to 1/2^{+}}|\zeta(\sigma+it)|=\infty$, $\mathbb{P}$-a.s.
\end{theorem}
\begin{remark}
Note that the second part of the Theorem implies that $\zeta(s, \omega)$ can't be extended analytically from the domain $\{Re s >1/2\}$ through the line $Re s=\frac{1}{2}$. 
\end{remark}

Let us recall that due to Riemann Conjecture (see, for example, \citep{riemann}) the classical $\zeta(\cdot)$ function has analytic continuation $\zeta(s)$, for $s \in \mathbb{C}$ with simple pole at $s=1$ and two groups of zeros: the so-called trivial zeros $s=-2, -4, \ldots-2r, \ldots$ and infinite set of zeros conjectured all to be on the critical line $Re s=1/2.$ 

\begin{remark}\label{extension}
[Extension of the Riemann Zeta Function]
Let us present the proof of the analytic continuation of the $\zeta(\cdot)$ into the critical strip $0<Re s<1$ without the use of the Gamma function and of the Poisson summation formula. 
For this, let us consider $\zeta(s)=\frac{1}{1^s}+\frac{1}{2^s}+\ldots$
Let us consider $A(s)=1-\frac{1}{2^s}+\frac{1}{3^s}-\ldots$\\
It is not difficult to prove that the series converges for $Re s>0$ and it is analytic in this region.
Then, we have that 
$$\zeta(s)=\sum_{k=1}^{\infty}\frac{1}{(2k-1)^s}+\frac{1}{2^s}\zeta(s).$$
That gives that $$\zeta(s)\left(1-\frac{1}{2^s}\right)=B(s).$$

Moreover, we have that 
$$A(s)=\sum_{k=2}^\infty \frac{1}{(2k-1)^{s}}-\frac{1}{2^s}\zeta(s)=B(s)-\frac{1}{2^s}\zeta(s).$$

Combining the previous two, we obtain that

$$\zeta(s)\left(1-\frac{1}{2^s}\right)=A(s)+\frac{1}{2^s}\zeta(s).$$

This implies that

$$\zeta(s)\left( 1-\frac{2}{2^s}\right)=A(s).$$
Thus, we have that

$$\zeta(s)=\frac{A(s)}{1-\frac{1}{2^{s-1}}}.$$

We obtain that the denominator has a pole for $s=1$ and the numerator is analytic for $Re s>0.$

\end{remark}

\section{Proof of the main result}

\begin{proof}[Proof of the main result]
We first prove part a).\\
Let us consider the regime $Re s >1.$ Then, we have that 

\begin{equation}
\zeta(s, \omega)=exp\left(\sum_{n=1}^{\infty}\frac{1}{p_n^s(\omega)}+\frac{1}{2}\sum_{n=1}^{\infty}\frac{1}{p_n^{2s}(\omega)}+\ldots \right).
\end{equation}
We note that since $p_n(\omega) \sim n\ln n,$ $\mathbb{P}$-a.s. the expression in the exponent excluding the first sum is analytic for $Re s > 1/2.$

Let us now investigate the convergence of the first sum. 

Let $S_1=\sum_{n=1}^{\infty}\frac{1}{p_n^s}=\frac{1}{2^s}+\frac{1}{3^s}+\sum_{n=3}^{\infty}\frac{1}{p_n^s(\omega)}.$ Let $S_1^{'}=\sum_{n=3}^{\infty}\frac{1}{p_n^s(\omega)}.$

We note that 
\begin{equation}
S_1^{'}=\sum_{n=3}^{\infty}\frac{\epsilon_n}{n^s},
\end{equation}
where $\epsilon_n$ are independent Bernoulli random variables with $\mathbb{P}(\epsilon_n=1)=\frac{1}{\ln n}, $ and $\mathbb{P}(\epsilon_n=0)=1-\frac{1}{\ln n}.$

In order to prove the $\mathbb{P}$-a.s. convergence of $S_1^{'}$, we use Kolmogorov's Two Series Theorem.

For this, we need to show the convergence of the series of expectations and the one of variances.

We have that 

\begin{equation}\label{expectation}
\mathbb{E}[S_1^{'}]=\sum_{n=3}^{\infty}\frac{1}{\ln n n^s}.
\end{equation}

Let $\phi(s)=\sum_{n=2}^{\infty}\frac{1}{ln n n^s}=\sum_{n=2}^{\infty}\frac{e^{-s ln n}}{ln n}, $ for $Re s>1.$
We thus have
$$\phi'(s)=-\sum_{n=2}^{\infty}e^{-s ln n}=1-\zeta(s).$$
We note that $\zeta(s)-1=\frac{1}{2^s}+\frac{1}{3^s}+\ldots$
We recall that by the analytic continuation of $\zeta(s)$ (see Remark \ref{extension}) we have that 
$$\zeta(s)=\frac{A(s)}{1-\frac{1}{2^{s-1}}}.$$

We note that $A(s)$ is analytic for $Re s>0$ (i.e. including the critical stripe and the critical line $Res =1/2$). 

Following \citep{Havil}, p.118, we obtain that the Riemann Zeta expansion around $s=1$ is given by

$$\zeta(s)=\frac{1}{s-1}+\sum_{n=0}^\infty \frac{(-1)^n}{n!}\gamma_n(s-1)^n.$$

The constants $\gamma_n$ are known as Stieltjes constants and these quantities can be computed via

$$\gamma_n=\lim_{m \to \infty}\left(\sum_{k=1}^m\frac{(ln k)^n}{k}- \frac{(ln (m))^{n+1}}{n+1} \right).$$

In our case in order to recover the function $\phi(s)$, we need to integrate the $\zeta(s)$. Following Remark \ref{extension} we obtain that $\zeta(s)$ can be analytically extended up to $Re s >0.$ Then, $\phi(s)$ as its integral can be extended in the same region. 
Due to the presence of the simple pole at $s=1$ for the Riemann Zeta function, its integral will involve the complex Logarithm function. Thus, we do not expect to have a unique analytic extensions of the $\phi(s)$ function up to $Re s >0.$  To see this, we provide some details in the following paragraphs.

First, we recall the exponential function with a complex parameter $$e^z=\sum_{n=0}^{\infty} \frac{z^n}{n !}.$$

\begin{definition} Let us consider $h(t)$ to be a continuous function defined on some interval $[a, b]$ satisfying $e^{h(t)}=\gamma(t)$. We call $h(t)$ a branch of $\log (z)$ along the curve $\gamma$.
\end{definition}

For a fixed $s_0$ in a domain in the complex plane, let $\gamma:[a, b] \rightarrow s_0$ be the constant curve. If $s=s_0$ is one solution of the equation $e^z=s_0$, then we obtain that $\left\{s_0+2k \pi i\right\}$ for $k \in \mathbb{Z}$, are all possible values of a branch of $\log h(s)$ at $s_0$.






Returning to the proof of our main result, we continue the analysis with the series of variances. For this, we obtain the following.

\begin{equation}
\operatorname{var}[S_1^{'}]=\sum_{n=3}^\infty \frac{1}{ln n}\left(1-\frac{1}{ln n}\right)\frac{1}{n^{2\sigma}},
\end{equation}
where we recall $s=\sigma+it.$
We note that this series converges for $Re s=\sigma >1/2$ (including at $s=1$). Thus, the Kolmogorov's Two-series Theorem gives the a.s. convergence for $Re s=\sigma >1/2.$


For part b) we use the following lemma \footnotemark{}. 

\footnotetext{Although we discussed a version of the proof of a similar lemma, we would like to acknowledge Davide Giraudo's detailed answer that we found on the math.stackexchange.com platform.}
\begin{lemma}
Suppose $\left(X_n\right)$ is a sequence of independent random variables with finite variances. We consider $S_n:=\sum_{j=1}^n X_j$ for the partial sums. Suppose
\begin{itemize}
\item 1. $\left(X_n\right)$ is uniformly bounded (i.e. there exists some $A>0$ such that for all $n \in \mathbb{N} \quad\left|X_n\right| \leq A$ a.s.)

\item 2. $\operatorname{var}\left(S_n\right) \rightarrow \infty$ as $n \rightarrow \infty$.

\end{itemize}

Then 

$$
\forall M>0 \quad \lim _{n \rightarrow \infty} \mathbb{P}\left(\left|S_n\right| \leq M\right)=0 $$

\end{lemma}

\begin{proof}[Sketch of the proof] One can show this result by considering independent copies $\left(X_j^{\prime}\right)_{j \geqslant 1}$ and $\left(X_j^{\prime \prime}\right)_{j \geqslant 1}$ of the sequence $\left(X_j\right)_{j \geqslant 1}$. The result follows via the bound $\left|X_j^{\prime}-X_j^{\prime \prime}\right| \leqslant 2A$, and the Lindenberg Central Limit Theorem. Indeed using these two one can show that for each positive $M$,

$$\lim _{n \rightarrow \infty} \mathbb{P}\left(\left|\sum_{j=1}^n X_j^{\prime}-X_j^{\prime \prime}\right| \leqslant M\right)=0 .$$

The following inclusion gives the desired result 

$$\left\{\left|\sum_{j=1}^n X_j^{\prime}\right| \leqslant \frac{M}{2}\right\} \cap\left\{\left|\sum_{j=1}^n X_j^{\prime \prime}\right| \leqslant \frac{M}{2}\right\} \subset\left\{\left|\sum_{j=1}^n X_j^{\prime}-X_j^{\prime \prime}\right| \leqslant M\right\}$$

\end{proof}

We observe that in our case \begin{equation}
\operatorname{var}[S_1^{'}]=\sum_{n=3}^\infty \frac{1}{ln n}\left(1-\frac{1}{ln n}\right)\frac{1}{n^{2\sigma}},
\end{equation} is divergent for $\sigma \leq 1/2$, for any $t \in \mathbb{R}.$ 
This gives that the series $S_1=\sum_{n=1}^{\infty}\frac{1}{p_n^s}$ is a.s. divergent when we approach the critical line on horizontal limits at any height. Thus, for any fixed $t \in \mathbb{R},$ we have $\mathbb{P}$-a.s. that
$$\limsup_{\sigma \to 1/2^{+}}|\zeta(\sigma+it)|=\infty.$$ 
 





In addition, we obtain that all the points on the critical line are singular points for $\zeta(s, \omega)$. Thus, we obtain that $\zeta(s, \omega)$
 cannot be analytically extended to any open domain that contains the critical line $Res=\sigma=1/2$ or subsets of it.

\end{proof}








\section{Problems in Number Theory in the language of Cram\'er quasi primes}

\subsection{A problem in Additive Number Theory}
In this subsection, we show the applicability of Probability Theory techniques such as the First Borel-Cantelli Lemma in the study of a problem from the topic of Additive Number Theory phrased in the language of quasi primes. 
Additive problems have been studied for a long time in Number Theory. One of the well-studied problems in this area concerns the representation of positive integers as a sum of four squares, which was solved by Lagrange in 1770. Other examples of problems include the ones in the area of 'Mixed-Power Problems'. Hardy and Littlewood have several conjectures on asymptotic formulas that arise from different representations of the integers. One of them concerns the number of representations of sufficiently large integer $n$ in the form $n=p+x^2+y^2$, where $p$ is a prime and $x$ and $y$ are integers. Linnick confirmed this conjecture in \citep{Linnik}.

Another famous problem (to the best of our knowledge, still open) proposed by Hardy and Littlewood concerns the asymptotic number of representations of a sufficiently large $n$ in the different form $n=p+m^2$, where $p$ is a prime and $m$ is an integer.

In the spirit of this problem, we consider a similar statement in the framework of the quasi primes. We prove the following result.





\begin{theorem}
For $n \geq 1$ let $A_n=\floor{e^{cn^\alpha}}$, for $\alpha<1/2$. Let $p_n(\omega)$ be the  quasi primes from the Cram\'er Model. Then, $\mathbb{P}$-a.s., there exists $n_0=n(\omega)$ such that for all $n \geq n_0$ one can find $i, j$ such that 
$$n=A_i+p_j(\omega).$$
\end{theorem}

\begin{proof}
Let us consider the event
$$E_n:=\{\text{there is no } i, \hspace{1mm} j \hspace{1mm} \text{such that} \hspace{1mm} n=A_i+p_j(\omega)\}. $$

Then, for $\nu(n):=\#\{A_i\leq n\},$ we have that 

\begin{align}
\mathbb{P}(E_n) \leq \prod_{i=0}^{\nu(n)}\left(1- \frac{1}{\ln (n-A_i)}\right)\leq e^{-\sum_{i=0}^{\nu(n)}\frac{1}{\ln(n-A_i)}} \leq e^{-\frac{\nu(n)}{\ln n}}.
\end{align}

We note that 
$$\#\{i: A_i \leq n \}=\#\{i: e^{ci^\alpha} \leq n\}=\left(\frac{\ln n}{c}\right)^{\frac{1}{\alpha}}$$

Thus, for $\alpha<1/2$, we have that

\begin{equation}\label{Bound}
\mathbb{P}(E_n) \leq e^{-c_1\ln^{1+\epsilon} n}.
\end{equation}

Since $\sum_n\mathbb{P}(E_n) <\infty$, applying the First Borel-Cantelli Lemma, we obtain that there exists $N_0(\omega)$ such that for $n >N_0(\omega)$ we have the desirable representation $n=A_i+p_j(\omega)$.
\end{proof}

We note that the case $\floor{m^k}$ for some $m$ and $k$ suitably chosen via $m^k=e^{ci^{\alpha}},$ i.e. $k=\floor{\frac{ci^{\alpha}}{ln m}}.$ 


\subsection{A problem on the number of primes in a given set}

The second classical problem is related to the number of primes in a given infinite set of integers. For the primes, there is the Dirichlet  classical result:
Let $a_n=a_0+dn,$ $n \geq 0$, be an arithmetic progression and $(a_0, d)=1$, that is $a_0$ and $d$ are co-primes. Then, there are infinitely many primes of the form $a_0+dn$ and there exists asymptotic formula for $\#\{p: p=a_0+dt, t \leq n\}.$ For more details we refer the reader to \citep{Apostol}.  Along this known result, there are many open problems in the area. The following two problems are especially interesting.\\
\textbf{Problem 1:} Let $\{p_n, n \geq 1\}=\{2, 3, 5, \ldots, \}$ be the sequence of primes. Let us consider $m_n=2^{p_n}-1$, $n \geq 1$, be the Mersenne numbers: $\{3, 7, 31, 127, \ldots\}$. The primes that are Mersenne numbers are called Mersenne primes. Many of the largest known primes are Mersenne primes. In particular, the largest known prime number $2^{82,589,933}-1$ is a Mersenne prime. The first composite Mersenne number is $m_{11}=23\cdot 89.$ It is conjectured that there are infinitely many primes among the Mersenne numbers.\\
\textbf{Problem 2:} Let $F_{n+1}=F_n+F_{n-1}$, $F_0=1$, $F_1=1$ be the sequence of Fibonacci numbers. One can find large subsets of composite Fibonacci numbers. It is conjectured that the numbers of primes among the Fibonacci numbers $\{F_n, n \geq 0\}$ is infinite.

In the spirit of the aforementioned problems in this area, we obtain in the framework of quasi-primes, as a corollary of the Borel-Cantelli Lemma, the following result. 

\begin{proposition}
Let $X=\{x_k, k \geq 1\}$ be a sequence of integers with $x_k \geq 3$, and $\sum_{k=1}^{\infty}\frac{1}{\ln x_k}=\infty. $ Then, $\mathbb{P}$-a.s. the set $X$ contains infinitely many Cram\'er quasiprimes. 
\end{proposition}

In order to show this, one has the probability of $x_k$ to be quasi-prime is $\frac{1}{ln x_k}$, then by the condition in the hypothesis of the proposition using the Second Borel-Cantelli Lemma we have that $\mathbb{P}$-a.s the event happens infinitely often, and the conclusion follows.

\section{Pseudoprimes with Local Symmetry}

In this section, we introduce a new random model of the primes that exhibit local symmetries and study the corresponding Random Riemann Zeta Function.

Let us consider $A_{n}, B_{n}$, $n\geq 1$ such that $A_n:=1+2B_1+\ldots+2B_{n-1}$ and $A_n+2B_n=A_{n+1}$.

Let us consider $C_n=A_n+B_n.$ Let $\xi_1$, $\xi_2$, ...$\xi_k$ be uniformly distributed on the left interval intersected with the integers, that is on the interval $[A_n, C_n]\cap \mathbb{Z}$. On the right symmetric part, we reflect the points. We use the notation $\xi_{(1)}, \ldots, \xi_{(k)}$ for the reflected points. Thus, using this method we create a correlated model. We emphasize that this is different from the Cram\'er Model of the quasi primes from before in which we have independence embedded in the model. Let us consider, $j \in \{1,\ldots 2k\}$ with the convention that $j=(j-k)$ for $j\geq k,$ then we have $P_{n,j}=C_n-\xi_j$, for $j < k$, and $P_{n,j}=C_n+\xi_{(j)}$, for $j  \geq k$. That is, using $C_n$ as a reference we have that the first $k$ pseudo-primes are obtained by extracting from $C_n$ the values $\xi_j$, while the next $k$ pseudo-primes by symmetrically adding to $C_n$ the values $\xi_{(j)}$.

We now recall that for the Riemann Zeta function, we have that 
\begin{equation}\label{Eulerformula}
\zeta(s)=\sum_{n=1}^{\infty}\frac{1}{n^s}=\prod_{n=1}^{\infty}\frac{1}{(1-\frac{1}{p_n^s})}, 
\end{equation}
where $\{p_n, n\geq 1\}=\{2,3,5,\ldots \}$ is the set of true primes.

As before, using the similarity with \eqref{Eulerformula}, we define using our new model of pseudo-primes, the corresponding random $\tilde{\zeta}=\tilde{\zeta}(s, \omega)$ function by the formula 

\begin{equation}
\tilde{\zeta}(s, \omega)=\prod_{j=1}^{\infty}\frac{1}{(1-\frac{1}{P^s_{n,j}(\omega)})}
\end{equation}

in the beginning for $Re s >1.$





\subsection{Analytic Continuation of the corresponding Random Riemann Zeta Function}

In  this subsection, we study the analytic continuation of the corresponding Riemann Zeta Function for our model of pseudo-primes. 

Let us consider the regime $Re s >1.$ Then, as before we have that 

\begin{equation}\label{expresNEW}
\tilde{\zeta}(s, \omega)=exp\left(\sum_{n=1, j=1, \ldots ,2k}^{\infty}\frac{1}{P_{n,j}^s(\omega)}+\frac{1}{2}\sum_{n=1, j=1, \ldots, 2k}^{\infty}\frac{1}{P_{n,j}^{2s}(\omega)}+\ldots \right).
\end{equation}

Let us focus on the first sum.
For a fixed $n \geq 1,$ we obtain the following.  

\begin{align}
&\frac{1}{(C_n-\xi_1)^s}+\cdots + \frac{1}{(C_n+\xi_{(k)})^s}\nonumber\\&=\frac{1}{C_n^s}\left(\left(1-\frac{\xi_1}{C_n}\right)^{-s}+\ldots+\left(1+\frac{\xi_{(k)}}{C_n}\right)^{-s}\right)
\end{align}
For $n$ large enough we have that

\begin{equation}
\frac{1}{C_n^s}\left(1-\frac{\xi_1}{C_n}(-s)+1+\frac{\xi_{(1)}}{C_n}(-s)+\ldots+1+\frac{\xi_{(k)}}{C_n}(-s)+O(C_n^{-2})\right)
\end{equation}

By sign cancellations due to the symmetry of the model, we further obtain

\begin{equation}
\frac{1}{C_n^s}\left(2k+2\frac{s(s+1)}{2}\frac{\xi_{1}^2}{C_n^2}+ \ldots   \right)
\end{equation}

Let us  consider $\xi_i=\floor{B_n\eta_i}$ with $\eta_i \sim U(0,1)$.

Then, using the Central Limit Theorem, we have that 

\begin{equation}\label{last}
\frac{2k}{C_n^s}+2\frac{s(s+1)}{2}\frac{\sum_{i=1}^k\xi_i^2}{C_n^{2+s}} = \frac{2k}{C_n^s}+\frac{ks(s+1)B_n^2}{6 C_n^{2+s}}+\frac{2\psi\sqrt{k}B_n^2}{\sqrt{45}C_n^{2+s}}+O(C_n^{-3-s})
\end{equation}
where $\psi \sim$ $\mathcal{N}(0,1)$. 

Let $\beta>0$. For $B_n=n^\beta$ we have that \eqref{last} is upper bounded by 

\begin{equation}
\frac{2k}{n^{s\beta}}+ \frac{ks(s+1)n^{2\beta}}{n^{\beta(2+s)}}+O(n^{-\beta(3+s)})
\end{equation}

For example, let us pick $k=\sqrt{n}$. Then, we get that for $\epsilon=Res>0$, we get that the series is convergent a.s. for $\beta>\frac{3}{2\epsilon}$. The condition gives also the convergence of the following series in \eqref{expresNEW}. This serves as an indication that once we improve the Cram\'er model by requiring certain local symmetries, the corresponding Random Riemann Zeta can be a.s. defined on a larger region than the region $Res>1/2$.

\bibliographystyle{plainnat}
\bibliography{bib}

\end{document}